\documentclass[10pt,reqno]{amsart}
\usepackage{geometry}
\usepackage{amssymb,bbold,mathrsfs,hyperref}

\newtheorem{theorem}{Theorem}
\newtheorem{lemma}[theorem]{Lemma}
\newtheorem{corollary}[theorem]{Corollary}

\newcommand{\R}{\mathbb{R}}
\DeclareMathOperator{\Op}{Op}
\newcommand{\m}{{\bf m}}
\DeclareMathOperator{\Tr}{Tr}

\newcommand{\C}{\mathbb{C}}

\title[Perturbation of resonances]{Perturbation of Ruelle Resonances and Faure Sj\"ostrand anisotropic space}
\author{Yannick Guedes Bonthonneau}
\address{Universit\'e de Rennes 1, CNRS, IRMAR - UMR 6625, F-35000 Rennes, France}
\begin{document}

\begin{abstract}
Given an Anosov vector field $X_0$, all sufficiently close vector fields also are of Anosov type. In this note, we check that the anistropic spaces described in \cite{Faure-Sjostrand-10}, \cite{Dyatlov-Zworski-16} can be chosen adapted to any smooth vector field sufficiently close to $X_0$ in $C^1$ norm.
\end{abstract}

\maketitle 

In this note, $M$ will denote a compact manifold of dimension $n$, and $X_0$ a smooth \emph{Anosov} vector field on $M$. That is, there exists a splitting
\begin{equation*}\label{eq:invariant-splitting}
TM = \R X_0 \oplus E^u_0 \oplus E^s_0.
\end{equation*}
The vector field $X_0$ never vanishes. The splitting is invariant under the flow of $X_0$, which is denoted $\varphi^0_t$. We also have constants $C,\beta>0$ such that for all $t\geq 0$,
\begin{equation}\label{eq:def-beta-C}
\| d \varphi^0_{-t |E^u_0} \| \leq C e^{-\beta t},\quad\text{ and }\quad \| d\varphi^0_{t|E^s_0} \| \leq C e^{-\beta t}.
\end{equation}
(here the norm is a fixed norm a priori, and one can check that although the constants depends on that choice of norm, their existence does not).

Starting with \cite{BKL02}, several authors have built some \emph{anisotropic spaces} of distributions to study the spectral properties of hyperbolic dynamics, of which Anosov flows are a prime example. This enables the study of so-called Ruelle-Pollicott resonances, originally defined using the (quite different) techniques of Markov partitions \cite{Ruelle-86,Pollicott-85}. They appeared as the poles of some zeta functions, popularized by Smale \cite{Smale-67}.

Since there is no canonical way to build those spaces, various constructions have been developped. In \cite{Baladi-Quest-17}, one can find a thorough review of the litterature, weighting the different advantages that each construction has to offer. Ten years ago, such a construction was introduced in \cite{Faure-Roy-Sjostrand-08}. In its functional analytic aspects, it relied on microlocal analysis tools. Provided one understands standard tools in that field, one can present these spaces in the following fashion
\begin{equation}\label{eq:}
\mathcal{H}_G : = \Op(e^{-G}) L^2(M).
\end{equation}
(here $\Op$ denotes a classical quantization, and $G \in S^{\log}(T^\ast M)$). Using these technique led to new developments. For example the description of the correspondance between classical and quantum spectrum in constant curvature at the level of eigenfunctions. Several results already obtained have also been reproved using these techniques. For example the meromorphic continuation of the dynamical zeta function, proved in general by \cite{Giulietti-Liverani-Pollicott-13}, and in the smooth case with microlocal methods by \cite{Dyatlov-Zworski-16}. 

So far, one particular aspect of the theory that was not reproduced by the microlocal techniques are perturbations. In this note, we will explain how this can be done. The first step will be to prove:
\begin{theorem}\label{thm:same-space-main}
Let $X_0$ be a $C^\infty$ Anosov vector field on $M$ a compact manifold. There exists $\eta>0$ such that the following holds. For any $R >0$, we can find $G \in S^{\log}(T^\ast M)$ such that for $\| X - X_0 \|_{\mathcal{C}^1} < \eta$ and $X$ $C^\infty$, the spectrum of $X$ acting on $\mathcal{H}_G$ is discrete in 
\[
\{ s\in \C\ |\ \Re s > - R\}.
\]
\end{theorem}

This discrete spectrum corresponds to the poles of the Schwartz kernel of the resolvent of the flow, so it does not really depend on the choice of $G$; it is called the \emph{Ruelle-Pollicott} spectrum of resonances of $X$. Our next result is
\begin{corollary}\label{cor}
Consider $\lambda_1,\dots,\lambda_N$ a finite set of resonances (counted with multiplicity) of $X_0$ and $\epsilon\mapsto X_\epsilon$ a $\mathcal{C}^\infty$ family of $\mathcal{C}^\infty$ vector fields perturbating $X_0$. Then, there is an $\epsilon_0 >0$, such that we can find continuous functions $\lambda_i(\epsilon)$ on $]-\epsilon_0,\epsilon_0[$, and an open set $\Omega\subset \C$ such that $\lambda_i(\epsilon)\in \Omega$, and the intersection of the spectrum of $X_\epsilon$ with $\Omega$ is the set $\{ \lambda_1(\epsilon), \dots, \lambda_N(\epsilon)\}$. Additionally, the resonances are smooth functions of $\epsilon$ where they do not intersect.
\end{corollary}

Considering only the dominating resonance (with maximal real part), this was proved in several context \cite{Butterley-Liverani-07,GL06}, and the proof for several resonances is the same as for only one; However we could not find a reference for this general statement. A similar statement should also hold in finite regularity (if $X_0$ is $C^r$ with $r>1$), but we are not using the best tools to tackle this type of question, so we ignore it altogether. 

While this paper was being elaborated, other authors were considering a similar question, to study the Fried's conjecture in \cite{Dang-Guillarmou-Riviere-Shen-18}; Discussing with them helped improved the present note. I also thank Y. Chaubet for the discussion that led to Lemma \ref{lemma:reg-proj}

\section{Microlocal proof of the main theorem}
\label{sec:1}

In what follows, we consider $X$ to be another vector field, assumed to be close to $X_0$ in $\mathcal{C}^1$ norm --- we will be clear when we use that assumption. We will denote by $\varphi_t$ the corresponding flow. The notation ${\ast}^{(0)}$ means that the statement is valid both for the object related to $X_0$ and $X$.

While the fundamental ideas are always similar, there are several ways to present the proof of Theorem \ref{thm:same-space-main}. We will use the version that was presented in \cite{Dyatlov-Zworski-16}. They use spaces of the form $\mathcal{H}_{rG}$. They assume that as $|\xi|\to +\infty$,
\[
G \sim \m(x,\xi)\log(1+|\xi|),
\]
$\m$ being a smooth $0$-homogeneous function on $T^\ast M$. The action of $X$ can be lifted to $T^\ast M$ by considering the Hamiltonian flow $\Phi_t(x,\xi)$ generated by $p= \xi(X)$:
\[
\Phi_t(x,\xi) = (\varphi_t(x), (d_x \varphi_t)^{-1 T} \xi).
\]
We denote by $X_\ast$ the generator of that flow. Since $p$ is $1$-homogeneous, $\Phi_t$ also is, and it acts on $S^\ast M = T^\ast M/\R^+ \simeq \partial T^\ast M$. The proof of \cite{Dyatlov-Zworski-16} is based on the use of the notion of sinks and sources. A \emph{sink} for $\Phi_t$ is a $\Phi_t$-invariant conical closed set $L \subset T^\ast M$ such that there exists a conical neighbourhood $U$ of $L\cap \partial T^\ast M$ and constants $C',\beta'$ such that for $\xi \in U, t>0$, 
\[
|\Phi_t(\xi)| \geq \frac{1}{C'} e^{\beta' t}|\xi|,
\] 
and $\Phi_t(U \cap \partial T^\ast M) \to L\cap \partial T^\ast M$ as $t\to +\infty$. A \emph{source} is a sink for $\Phi_{-t}$. One can define subbundles of $T^\ast M$ of particular interest by setting
\[
E^\ast_{00} = (E^u_0 \oplus E^s_0)^\perp, \quad E^\ast_{u0} = (\R X_0 \oplus E^s_0)^\perp,\quad E^\ast_{s0} = (E^u_0 \oplus \R X_0)^\perp.
\]
One gets the decomposition
\begin{equation}\label{eq:dual-splitting}
T^\ast M = E^\ast_{00}\oplus E^\ast_{u0} \oplus E^\ast_{s0}.
\end{equation}
The bundles $E^{\ast}_{u,s,0}$ are invariant under the corresponding flow $\Phi_t$. One can check that $E^\ast_{u0}$ is a sink, and $E^\ast_{s0}$ is a source for $\Phi_t^0$.

The proofs of Propositions 3.1 and 3.2 in \cite{Dyatlov-Zworski-16} require the following input.
\begin{enumerate}
	\item The flow $\Phi_t$ admits a source $E^\ast_u$ and a sink $E^\ast_s$. Both are contained in $\{ p = 0\}= X^\perp$.
	\item Given neighbourhoods $U_u$, and $U_s$ as in the definition of sinks and sources, there is a $T>0$ such that $\Phi_t( U_s^c\cap X^\perp) \subset U_u $ and $\Phi_{-t}(U_u^c\cap X^\perp) \subset U_s$ for all $t> T$.
	\item the weight $\m$ satisfies
\begin{equation}\label{eq:condition-m}
\m = 1\text{ on }U_u ,\quad \m = - 1 \text{ on }U_s,\quad X_\ast \m \geq 0.
\end{equation}
\end{enumerate}
(we have inverted $E^\ast_u$ and $E^\ast_s$ from their article because we consider $X- s$ with $\Re s> - C$ instead of $- i X - \lambda$ with $\Im \lambda > - C$).

Provided conditions (1)-(3)	 are satisfied, the proof in \cite{Dyatlov-Zworski-16} applies, and we obtain the following. For any $R$, one can find an $r>0$ such that the spectrum of $X$ on $\mathcal{H}_{rG}$ is discrete in $\{ \Re s > - R\}$. For our purposes this is not sufficient because we cannot let the $r$ depend on $X$ as long as $X$ is close enough to $X_0$. However, one can give a rough estimate on the value of $r$.

Following the proof of Proposition 2.6 in \cite{Dyatlov-Zworski-16}, one finds that it applies (and this is the condition for the proof of Propositions 3.1 and 3.2 to also apply) if $r$ satisfies
\[
r c > C_0 + \sup_{T^\ast M} \frac{ X_\ast |\xi|}{|\xi|}.
\]
Let us explain how one obtains the constants $c$ and $C_0$. The $c$ comes from Lemma C.1 in \cite{Dyatlov-Zworski-16}. For some $T_1>0$, one has $|\Phi_{T_1}(\xi)|> 2 |\xi|$ for all $\xi \in U_u$, and $c$ is defined by
\[
\int_0^{T_1} |\Phi_t(x,\xi)| dt \leq \frac{1}{c} |\xi|, \ \forall \xi\in T^\ast M.
\]
We deduce that 
\[
c = \Lambda e^{-\Lambda T_1} = \Lambda \left(\frac{C'}{2} \right)^{\frac{\Lambda}{\beta'}},
\]
where $|\Phi_t(x,\xi)| \leq e^{\Lambda |t|} |\xi|$ for $\xi \in T^\ast M$, and $C',\beta'$ are the constants in the definition of being a sink for $E^\ast_u$. The constant $\Lambda$ can be estimated directly as $\Lambda \leq \| dX\|_{L^\infty}$ by usual estimates.

The constant $C_0$ in the proof of Proposition 2.6 is chosen at the end of the proof to absorb some other terms. To be more precise, $C_0$ has to satisfy $C_0 > 4 C_1$, where
\[
\Re \langle (X-s) u, u \rangle \leq C_1 \| u\|_{L^2},
\]
for all $u \in L^2(M)$. We get $C_1 = \| \mathrm{div} X\|_{L^\infty}/2 - \Re s$. Finally, we get
\[
r > r_X(s):= \frac{1}{\Lambda}\left(\frac{2}{C'} \right)^{\frac{\Lambda}{\beta'}}\left( 2 \| \mathrm{div} X\|_{L^\infty} - 4 \Re s  +\sup_{T^\ast M} \frac{ X_\ast |\xi|}{|\xi|} \right) 
\]
We call $r_X(s)$ the \emph{minimal strength}, and this is called the \emph{Threshold condition}. The proof of Theorem \ref{thm:same-space-main} will thus be done if we can prove the following lemma:
\begin{lemma}\label{lemma:main}
There are conical open sets $U_u$ and $U_s$, $\m\in C^\infty(S^\ast M)$ and $\eta>0$ such that whenever $\| X - X_0 \|_{C^1} < \eta$, $X$ satisfies (1), (2) and (3), with $U_u$ (resp $U_s$) an admissible neighbourhood for $E^\ast_u$ (resp. $E^\ast_s$). Additionally, the constants $C'$ and $\beta'$ satisfy
\[
C' > \frac{C}{2},\quad \beta' > \frac{\beta}{2},
\]
where $C,\beta$ are the constants defined in equation \eqref{eq:def-beta-C}. The minimal strength $r_X(s)$ is thus bounded uniformly.
\end{lemma}

Except for the construction of the weight function $\m$, our Lemma can probably be seen as a corollary of structural stability for Anosov flows \cite{DeLaLlave-Marco-Moryon-86}. However, we can give a full proof directly, which is quite elementary.

A last remark is that to work with the spaces $\Op(e^{-rG})H^k(M)= \mathcal{H}_{rG + k\log\langle \xi\rangle}$, the same proof applies. The Threshold conditions becomes $r+k > r_X(s)$, and $r-k > r_X(s)$, i.e. $r> r_X(s) + |k|$.

\section{Building the weight function}

This section is devoted to proving Lemma \ref{lemma:main}. We are now working in $S^\ast M$. We start by building a weight adapted to $X_0$, following the strategy for \cite{Faure-Sjostrand-10}. We denote by $\Phi_t^{(0),\infty}$ the flow $\Phi_t^{(0)}$ projected on $S^\ast M$. Recall the following:
\begin{lemma}
The bundles $E^u_0$ and $E^s_0$ are continuous. In particular, there is a positive lower bound for the angle between any two bundles in $E^u_0$, $E^s_0$, $E^0_0= \R X_0$.
\end{lemma}
This is a very usual lemma, contained in theorem 3.2 in \cite{Hirsch-Pugh-70}.

\begin{lemma}\label{lem:neighbourhoods}
Given $\epsilon>0$, there exist $T>0$ such that when $t\geq T$, for $\xi \in S^\ast M$,
\begin{align*}
d(\xi, E^\ast_{s0})>\epsilon &\Rightarrow d( \Phi^{0,\infty}_t(\xi), E^\ast_{u0}\oplus E^\ast_{00}) \leq \epsilon\\
\intertext{and}
d(\xi, E^\ast_{u0}\oplus E^\ast_{00})>\epsilon &\Rightarrow d( \Phi^{0,\infty}_{-t}(\xi), E^\ast_{s0}) \leq \epsilon
\end{align*}
\end{lemma}

In the proof, the constant $\infty > C>0$ may change at every line, but it is always controlled by the lower bound on the angles between the bundles.
\begin{proof}
Since the angles between the distributions $E^\ast_{u0,s0,00}$ are bounded by below, we have a constant $C>0$ such that if $\xi = \xi_s + \xi_u+\xi_0$ is the decomposition with respect to \eqref{eq:dual-splitting},
\begin{equation*}
d(\xi, E^\ast_{s0}) > \epsilon \rightarrow |\xi_u + \xi_0|> C \epsilon. 
\end{equation*}
Next, we also have
\begin{equation*}
|\xi_0| = \sup_{|u| = 1} \xi_0(u) = \sup \{ \xi( \alpha X_0)\ |\ \|\alpha X_0 +u' \| = 1,\ u'\in E^u_0 \oplus E^s_0\}.
\end{equation*}
By the bounded angle property, if $\| \alpha X_0 + u' \| = 1$, we have $|\alpha| \leq C$. We deduce ($X_0$ never vanishes !)
\begin{equation*}
\frac{1}{C}|\xi(X_0)| \leq |\xi_0| \leq C |\xi(X_0)|.
\end{equation*}
Since by definition, $\varphi^0_{t\ast}(X_0) = X_0$, $(\Phi^0_t(\xi))(X_0(\varphi_t(x)) = \xi(X_0)$, we deduce that
\begin{equation*}
|\Phi_t^0 (\xi) | \geq \frac{1}{C}(|\Phi_t^0(\xi_u)|+  |\Phi_t^0(\xi_0)|) \geq ( |\xi_u|+|\xi_0|)/C^2 \geq  \epsilon/ C^3.
\end{equation*}
In particular, again using the bounded angle property, 
\begin{equation*}
d(\Phi^{0,\infty}_t(\xi), E^\ast_{u0}\oplus E^\ast_{00}) \leq C\frac{ |\Phi_t^0(\xi_s)|}{|\Phi_t^0(\xi)|}\leq \frac{C}{\epsilon}e^{-\beta t}.
\end{equation*}
To obtain the first implication, taking $T \geq  (\log C - 2 \log \epsilon)/\beta$ will suffice, and for the second implication, a similar reasoning will apply. 
\end{proof}

Since $E^{\ast}_{s0}$ and $E^\ast_{u0}\oplus E^\ast_{00}$ are continuous and transverse, for $\epsilon>0$ small enough, 
\[
\{ \xi\ |\ d(\xi, E^\ast_{s0}) \leq \epsilon\} \cap \{ \xi\ |\ d(\xi, E^\ast_{u0}\oplus E^\ast_{00}) \leq \epsilon\} = \emptyset.
\]
For such an $\epsilon>0$, we pick $m_0 \in \mathcal{C}^\infty(S^\ast M)$ taking values in $[0,1]$, such that 
\begin{equation*}
\{ \xi\ |\ d(\xi, E^\ast_{s0}) \leq \epsilon\} \subset \{ m_0 = 0\} \text{ and }\{ \xi\ |\ d(\xi, E^\ast_{u0}\oplus E^\ast_{00}) \leq \epsilon\} \subset \{ m_0 = 1 \}.
\end{equation*}
Then, let
\begin{equation*}
m := \int_{-T}^T m_0 \circ \Phi_t^{0,\infty} dt.
\end{equation*}
The derivative of $m$ along the flow $\Phi^{0,\infty}_t$ is
\begin{equation*}
F := m_0\circ \Phi^{0,\infty}_T - m_0\circ\Phi^{0,\infty}_{-T}.
\end{equation*}
(it is a smooth function). Consider a point $\xi \in S^\ast M$ such that
\begin{equation*}
F(\xi) = 0.
\end{equation*}
Assume that $m_0(\xi) \in ]0,1[$. Then by definition of $T$, we get that $m_0(\Phi^{0,\infty}_T(\xi)) = 1$ and $m_0(\Phi^{0,\infty}_{-T}(\xi)) = 0$ so that $F(\xi)$ cannot vanish. This contradiction implies that $m_0(\xi)$ is either $0$ or $1$. By symmetry, we can assume that $m_0(\xi) = 0$. 

In that case, by Lemma \ref{lem:neighbourhoods}, we get that $m_0(\Phi^{0,\infty}_{T}(\xi)) = m_0(\Phi^{0,\infty}_{-T}(\xi)) = 0$. Then, using Lemma \ref{lem:neighbourhoods} again, we get that $m_0(\Phi^{0,\infty}_{-t}(\xi)) = 0$ for $t\geq 0$. In particular, we have
\begin{equation*}
m(\xi) = \int_0^T m_0(\Phi^{0,\infty}_t(\xi)) dt < T.
\end{equation*}
Conversely, if we assumed that $m_0(\xi) = 1$, we would find that
\begin{equation*}
m(\xi) > T.
\end{equation*}

Now, we deduce the crucial lemma
\begin{lemma}
There are constants $\varepsilon,\delta>0$ such that 
\begin{equation*}
\{ | m - T | < \varepsilon \} \subset \{ F(\xi) \geq \delta\}.
\end{equation*}
\end{lemma}

\begin{proof}
Since $\{ F(\xi) = 0 \}$ is compact, the $\inf$ of $|m-T|$ is attained and by the preceding argument, is strictly positive. Denote it by $2\delta$. 

Now, consider
\begin{equation*}
\ell(\varepsilon) := \sup\{ d(\xi, \{ F = 0\})\ |\ F(\xi) \leq \varepsilon\}.
\end{equation*}
By continuity of $F$ and compacity of $S^\ast M$, we have that $\ell(\varepsilon) \to 0$ as $\varepsilon \to 0$. Finally, since $m$ is continuous, we have an $\varepsilon'>0$ such that whenever $d(\xi, \{F=0\}) < \varepsilon'$, $|m-T| > \delta$. It suffices then to take $\varepsilon$ such that $\ell(\varepsilon) < \varepsilon'$.
\end{proof}

Let us now come back to our perturbation problem. Consider $X = X_0 + \lambda V$ with $V$ a smooth vector field with $\| V \|_{\mathcal{C}^1} \leq 1$, and $\lambda>0$ small. The vector fields generating $\Phi^{(0)}_t$ are the hamiltonian vector fields of the principal symbols of $- iX_{(0)}$, which are $\xi(X_{(0)})$. In particular, they involve the first derivative of the vector fields $X_{(0)}$, so that they are $\mathcal{O}(\lambda)$-$\mathcal{C}^0$-close, with a constant depending on $V$ only through $\|V\|_{\mathcal{C}^1}$.

Then the corresponding vector fields on $S^\ast M$, $X^\infty_0$ and $X^\infty= X^\infty_0+ \lambda V^\infty$ that generate the boundary flows $\Phi^{(0),\infty}_t$ also are $\mathcal{O}(\lambda)$-$\mathcal{C}^0$-close, since they are the projection on $T(S^\ast M)$ of the previous hamiltonian vector fields. Observe that
\begin{equation*}
X^\infty m  = X^\infty_0 m + \lambda \int_{-T}^T V^\infty (m_0\circ\Phi^{0,\infty}_t) dt
\end{equation*}
Since $\|V^\infty\|_{L^\infty}=\mathcal{O}(1)$, the integral in the RHS is of size $\mathcal{O}(T \lambda)$. Now we will use our previous arguments. Let $\chi$ be a smooth function on $\R$, such that $\chi$ is constant equal to $-1$ in $]-\infty,-\varepsilon]$ and constant equal to $1$ in $[\varepsilon, +\infty[$, and strictly increasing in $[-\varepsilon,\varepsilon]$. Let
\begin{equation*}
{\bf{m}} := \chi(m-T).
\end{equation*}
We get directly that
\begin{equation*}
X_0^\infty {\bf{m}} \geq 0
\end{equation*}
But we also have
\begin{equation*}
X^\infty{\bf{m}} = \chi'(m-T) \left[ X^\infty_0 m + \lambda \int_{-T}^T V^\infty (m_0\circ\Phi^{0,\infty}_t) dt \right]
\end{equation*}
On the support of $\chi'(m-T)$, we have $X^\infty_0 m \geq \delta$. In particular, with $\lambda$ smaller than $\eta_0=\delta/CT$ with $C>0$ large enough, we get that $X^\infty {\bf{m}} \geq 0$.

\begin{lemma}\label{lemma:sources}
There is an $0< \eta \leq \eta_0$ such that the following holds. Let $X$ be a $C^1$ vector field such that $\| X-X_0\|_{\mathcal{C}^1} \leq \eta$. Then $\Phi_t$, when restricted to $\{\xi(X)=0\}$ has a source $E^\ast_s$, contained in $\{ {\bf m} = - 1\}$ (resp. a sink $E^\ast_u$ contained in $\{ {\bf m} =  1\}$.

There exist a $T>0$ --- only depending on $\eta$ --- such that $\Phi_t( \{ {\bf m} > -1\}) \subset \{ {\bf m}= 1\}$ for all $t\geq T$. (and the converse statement in negative time also holds). 
\end{lemma}

\begin{proof}
Following the arguments in the proof of Lemma \ref{lem:neighbourhoods}, we can find open cones $V_1\subset V_2$ containing $E^\ast_{u0}$ as small as desired, such that $\overline{V_1} \subset V_2$, and for some $t>0$, $\Phi^{0}_t (V_2) \subset V_1$, and $|\Phi^{0}_t (\xi)| > 3 |\xi|$ for all $\xi \in V_2$. 

Since $\| V\|_{\mathcal{C}^1} \leq 1$, for $\lambda$ small enough, we find that $\Phi_t(V_2) \subset V_1'$, where $V_1\subset V_1' \subset V_2$, and $\overline{V_1'} \subset V_2$. We also get that $|\Phi_t (\xi)| > 2 |\xi|$ for all $\xi \in V_2$. Now, we let
\begin{equation*}
E^\ast_{u}(x) := \{ \xi \in V_2(x)\ |\ \Phi_{-t}(\xi)\in V_2 \text{ for all }t\geq 0 \}.
\end{equation*}
Since we can write this as a decreasing intersection of compacts sets (compact in $T^\ast_x M \cup S^\ast_x M$), it is non empty, closed, and it is a cone by linearity of $\Phi_{t}$. We deduce that $E^\ast_{u}$ is a sink for $\Phi_t$

Likewise, we can find similar cones around $E^\ast_{s0}$ for negative times, and obtain that the corresponding $E^\ast_{s}$ is a source. 

For the points that are neither in the source nor in the sink, we can directly use Lemma \ref{lem:neighbourhoods}. Finally, since we could choose the neighbourhood $V_2$ as small as desired, since $E^\ast_{u} \subset V_2$, and since ${\bf m} = +1$ in a neighbourhood of $E^\ast_{u0}$, the proof is complete.
\end{proof}

\section{Perturbation of resonances}

The main step in the proof of \cite{Dyatlov-Zworski-16} is to prove the following (their Proposition 3.4). They use semi-classical spaces
\[
\mathcal{H}_{h,rG} := \Op_h(e^{-rG})L^2(M),
\]
$\Op_h$ being now a semi-classical quantization with small parameter $h>0$. This space is the same as $\mathcal{H}_{rG}$, albeit with a different, equivalent norm depending on $h>0$. We fix $Q$ a positive self-adjoint pseudor which is microsupported and elliptic around the zero section.
\begin{lemma}\label{lemma:control}
Under the assumptions from section \ref{sec:1}, for $0<h\leq h_0$, for $r> r_X(\Re s)$ and $|\Im s| \leq h^{-1/2}$, $hX - Q - s$ is invertible on $\mathcal{H}_{h,rG}$ and
\[
\| (h X - Q - s)^{-1} \| \leq \frac{C}{h}.
\]
\end{lemma}
One can check that the constants $h_0$ and $C$ can be estimated as
\[
\frac{1}{h_0}, C \leqslant \| X \|_{C^{N(r)}}.
\]
where the number of derivatives $N(r)$ only depends on $r$ and the dimension. Now, we consider a smooth family $X_\epsilon$ of vector fields perturbating $X_0$. From Lemma \ref{lemma:main}, we can find $\epsilon_0>0$, and $r(s)$ a non increasing function of $\Re s$, so that $r(s) \geq r_{X_\epsilon}(s)$ for all $\epsilon \in ]-\epsilon_0,\epsilon_0[$. Using Lemma \ref{lemma:control}, we get a uniform control Lemma:
\begin{lemma}
Consider $X_\epsilon$ a smooth family of vector fields perturbating $X_0$. Then, there is $\epsilon_0>0$ such that the following holds. Given any $s_0 >0$, $k$ and $r>r(s_0)+|k|$, there is $h_k>0$ such that for $0<h<h_k$, $\Re s > \Re s_0$, $|\Im s|<h^{-1}$ and $|\epsilon|<\epsilon_0$,
\[
X_\epsilon - h^{-1}Q -s
\]
is invertible on $\mathcal{H}_{h,rG+k\log\langle \xi \rangle}$. The inverse is then bounded as $\mathcal{O}(1)$ independently of $\epsilon$.
\end{lemma}

We observe that
\begin{equation*}
(X-s) (X-h^{-1}Q-s)^{-1} = \mathbb{1} + h^{-1} Q (X-h^{-1}Q-s)^{-1}.
\end{equation*}
Let us denote $D(X, s) = h^{-1} Q(X-h^{-1} Q -s)^{-1}$. Since $Q$ is smoothing, $D(X,s)$ is Trace class. In particular, the resonances of $X$ in $\Omega_{h,s_0}=\{ \Re s > \Re s_0,\ |\Im s|< h^{-1/2}\}$ are the $s$'s such that the Fredholm determinant
\begin{equation*}
F(X,s) := \det(\mathbb{1} + D(X,s)),
\end{equation*}
vanishes. For fixed $X$ $\mathcal{C}^1$-close to $X_0$, $F(X,\cdot)$ is a holomorphic function in $\Omega_{h,s_0}$. Now, we consider a smooth family $X_\epsilon$. We observe that 
\begin{equation*}
\partial_\epsilon F(X_\epsilon, s) = F(X_\epsilon, s) \Tr \left[ (1+D(X_\epsilon,s))^{-1} \partial_\epsilon D (X_\epsilon, s)\right]
\end{equation*}
and
\begin{equation*}
\partial_\epsilon D(X_\epsilon, s) = h^{-1} Q (X_\epsilon-h^{-1}Q-s)^{-1} \partial_\epsilon X_\epsilon (X_\epsilon-h^{-1}Q-s)^{-1}
\end{equation*}
Since in the $\Tr$ in the formula for $\partial_\epsilon F$, the operators are smoothing, this trace does not depend on the Sobolev space with respect to which we are taking the trace. We will denote by $\| \cdot \|_{\Tr}$ the norm on the space $\mathcal{L}^1(\mathcal{H}_{h,rG}, \mathcal{H}_{h,rG})$. Additionally, 
\begin{equation*}
\| F(X_\epsilon, s) (1+D(X_\epsilon,s))^{-1}\| \leq  \exp 2 \| D \|_{\Tr}+ 1,
\end{equation*}
is uniformly bounded (see equation B.5.15 in \cite{Dyatlov-Zworski-book}). From the formula, we deduce that $\partial_\epsilon F(X_\epsilon, s)$ defines a holomorphic function in the $s$ parameter. In particular, to obtain estimates on its derivatives in $s$, it suffices to estimate
\begin{align*}
\| \partial_\epsilon D (X_\epsilon, s) \|_{\Tr} &\leq Ch^{-1} \| Q  (X_\epsilon-h^{-1}Q-s)^{-1} \partial_\epsilon X_\epsilon \|_{\Tr} \\
					&\leq Ch^{-1}\|\partial_\epsilon X_\epsilon\|_{\mathcal{H}_{h,rG}\to \mathcal{H}_{h,rG-\log\langle\xi\rangle}} \| Q \|_{\mathcal{L}^1(\mathcal{H}_{h,rG - \log \langle \xi\rangle},\mathcal{H}_{h,rG})}\\
					&= \mathcal{O}(h^{-2-n}).
\end{align*}
(with a constant $C$ changing at every line). Here, we needed $(X-h^{-1}Q-s)^{-1}$ to be bounded on $\mathcal{H}_{h,rG -\log \langle\xi\rangle}$, so that the computation is only valid for $h< h_1$.

By an induction argument, we obtain that for $h< h_k$,
\begin{equation*}
|\partial_\epsilon^k F(X_\epsilon, s)|\leq C_k h^{-(2+n)k},
\end{equation*}
so that $\epsilon,s\mapsto F(X_\epsilon, s)$ is valued in 
\begin{equation*}
\mathcal{C}^k( ]-\epsilon_0, \epsilon_0[, \mathscr{O}( \Re s > -\Re s_0,\ |\Im s|\leq h^{-1/2})).
\end{equation*}

Since the resonances do not depend on the choice of space, it does not matter if we were using $\mathcal{H}_{rG}$ or $\mathcal{H}_{h,rG + k\log\langle\xi\rangle}$. 

We consider now a finite sequence $\lambda_1$, ..., $\lambda_N$ of resonances of $X_0$ counted with multiplicity. For $\delta>0$ let $\Omega_\delta:=\{ s \ |\ d(s, \{ \lambda_1, \dots,\lambda_N\}) < \delta\}$. Consider some $\delta>0$, and some $s_0$ so that $r(s_0)>r(\lambda_i-\delta)$ for all $i$. Then for $\delta,h>0$ small enough, $\Omega_\delta\subset \Omega_{h,s_0}$, and for some $\delta'>0$, $|F(X_0, s)|>\delta'$ for $s\in \partial\Omega_\delta$. In particular, there is $\epsilon_0'<\epsilon_0$ such that $|F(X_\epsilon,s)| > \delta'/2$ for $d(s, \{ \lambda_1, \dots,\lambda_N\}) = \delta$ and $|\epsilon|<\epsilon_0'$.

By the Rouch\'e theorem, we deduce that the zeroes of $F(X_\epsilon,s)$ in $\Omega_\delta$ can be parametrized by continuous functions, which are $\mathcal{C}^\infty$ when the resonances are simple. This proves Corollary \ref{cor}.

As a remark, let us consider spectral projectors. We come back to classical operators ($h=1$). Retaking the notations of Corollary \ref{cor}, we consider an open set $\Omega$, and $\lambda_1(\epsilon),\dots,\lambda_N(\epsilon)$ continuous functions such that the spectrum of $X_\epsilon$ intersected with $\Omega$ is exactly $\{\lambda_1(\epsilon),\dots,\lambda_N(\epsilon)\}$ counted with multiplicity. Next, we pick a closed curve $\gamma$ contained in $\Omega$, assuming that it does not contain any $\lambda_i(0)$. Then this remains true on some interval $|\epsilon|< \epsilon'$. As a consequence, for each such $\epsilon$, the following operator is well defined
\[
\Pi_\gamma(\epsilon)=\int_\gamma (X_\epsilon-s)^{-1} ds.
\]
By usual arguments, one can show that for $r > r(s_0)+|k|$ where $s_0 = \inf \Re \gamma(t)$, it is a bounded projector in $\mathcal{L}(\mathcal{H}_{rG+k\log\langle \xi\rangle})$. Using the resolvent formula, we get for $|\epsilon|<\epsilon'$,
\[
\partial_\epsilon \Pi_\gamma = \int_\gamma (X_\epsilon-s)^{-1} \partial_\epsilon X_\epsilon (X_\epsilon-s)^{-1} ds,
\]
which is a bounded operator from $\mathcal{H}_{rG}$ to $\mathcal{H}_{rG - \log\langle\xi\rangle}$ and also from $\mathcal{H}_{rG+\log\langle\xi\rangle}$ to $\mathcal{H}_{rG}$. On the other hand, since $\Pi_\gamma$ has finite rank and is a projector, its derivatives also have finite rank. This comes form the relation
\[
\partial_\epsilon \Pi = \Pi \partial_\epsilon \Pi + (\partial_\epsilon \Pi) \Pi.
\]
Since it has finite rank, the range of $\partial_\epsilon\Pi$ is contained in $\mathcal{H}_{rG}$, and it is bounded on that space (provided $r> r_0 + 1$). By induction, we deduce
\begin{lemma}\label{lemma:reg-proj}
Consider a closed curve $\gamma$ such that for $|\epsilon|<\epsilon_0$, no resonance crosses $\gamma$. Then, for $r> r(s_0) + k+1$, $\epsilon\mapsto\Pi_\gamma(\epsilon)$ is a $C^k$ family of bounded operators on $\mathcal{H}_{rG}$.
\end{lemma} 

Since the generator is elliptic in the direction of the flow, one could probably refine this statement to show that $\Pi_\gamma$ gains regularity in that direction, however we will not investigate this here.


\end{document}